\documentclass{amsart}

\usepackage{amssymb,amsmath,enumerate}

\usepackage[shortlabels]{enumitem}
\usepackage{colonequals}
\usepackage[nobysame]{amsrefs}

\usepackage{eucal}

\usepackage{hyperref}
\hypersetup{%
  bookmarksnumbered=true,%
  colorlinks=true,%
  linkcolor=blue,%
  citecolor=blue,%
  filecolor=blue,%
  menucolor=blue,%
  urlcolor=blue,%
  bookmarksopen=true,%
  bookmarksdepth=2,%
  pageanchor=true}

\makeatletter
\@namedef{subjclassname@2020}{\textup{2020} Mathematics Subject Classification}
\makeatother

\usepackage{lineno}

\usepackage{tikz-cd}
\usetikzlibrary{decorations.pathmorphing}

\newcommand{\ad}{\operatorname{ad}}
\newcommand{\aqch}[4]{\operatorname{D}^{#1}(#3/#2;#4)}
\newcommand{\aqh}[4]{\operatorname{D}_{#1}(#3/#2;#4)}
\newcommand{\atiyah}[3][{}]{\operatorname{At}^{#1}_{#2}(#3)}
\newcommand{\atiyahclass}[2][{}]{\operatorname{At}_{#1}^{#2}}

\newcommand{\cat}[1]{{\mathsf{#1}}}
\newcommand{\cone}{\operatorname{cone}}
\newcommand{\coker}{\operatorname{coker}}
\newcommand{\cotan}[1]{\operatorname{L}(#1)}
\newcommand{\cotam}[2]{\operatorname{C}_{#1}(#2)}
\newcommand{\dcat}[1]{\cat{D}(#1)}
\newcommand{\dsing}[1]{\cat{D}_{\mathrm{sg}}(#1)}
\newcommand{\dual}[1]{{#1}^{\scriptscriptstyle{\vee}}}

\newcommand{\env}[2]{{#2}_{#1}^{\operatorname{e}}}
\newcommand{\Ext}[4][{}]{\operatorname{Ext}^{#1}_{#2}(#3,#4)}

\newcommand{\fdim}{\operatorname{flat\,dim}}
\newcommand{\pdim}{\operatorname{proj\,dim}}

\newcommand{\hcm}[4][{*}]{\operatorname{HH}^{#1}(#3/#2;#4)}
\newcommand{\hh}{\operatorname{H}}
\newcommand{\HH}[2]{\operatorname{H}_{#1}(#2)}
\newcommand{\Hom}[3]{\operatorname{Hom}_{#1}(#2,#3)}
\newcommand{\Ker}{\operatorname{Ker}}
\newcommand{\rank}{\operatorname{rank}}

\newcommand{\RHom}[3]{\operatorname{RHom}_{#1}(#2,#3)}

\newcommand{\lotimes}[1]{\otimes^{\operatorname{L}}_{#1}}

\newcommand{\spec}{\operatorname{Spec}}
\newcommand{\susp}{{\mathsf{\Sigma}}}
\newcommand{\Tor}[4][{}]{\operatorname{Tor}_{#1}^{#2}(#3,#4)}
\newcommand{\ve}{\varepsilon}
\newcommand{\vf}{{\varphi}}
\newcommand{\wh}{\widehat}
\newcommand{\wt}{\widetilde}

\newcommand{\eps}{\varepsilon}

\newcommand{\lra}{\longrightarrow}
\newcommand{\xra}{\xrightarrow}

\newcommand{\ges}{\geqslant}

\newcommand{\fm}{{\mathfrak m}}
\newcommand{\fp}{{\mathfrak p}}

\newtheorem{theorem}[subsection]{Theorem}
\newtheorem{proposition}[subsection]{Proposition}
\newtheorem{lemma}[subsection]{Lemma}
\newtheorem{corollary}[subsection]{Corollary}
\newtheorem{theorema}{Theorem}

\theoremstyle{definition}
\newtheorem{chunk}[subsection]{}

\theoremstyle{remark}

\newtheorem*{ack}{Acknowledgements}

\numberwithin{equation}{subsection}

\begin{document}

\title[Cotangent complex]{Rigidity properties of the cotangent complex}

\author[B.~Briggs]{Benjamin Briggs}
\address{Department of Mathematics,
University of Utah, Salt Lake City, UT 84112, U.S.A.}
\email{briggs@math.utah.edu}

\author[S.~B.~Iyengar]{Srikanth B.~Iyengar}
\address{Department of Mathematics,
University of Utah, Salt Lake City, UT 84112, U.S.A.}
\email{iyengar@math.utah.edu}

\thanks{Partly supported by NSF grant DMS-2001368 (SBI)}

\date{23 January 2022}


\begin{abstract} 
This work concerns maps $\varphi \colon R\to S$ of commutative noetherian rings, locally of finite flat dimension.  It is proved that the Andr\'e-Quillen homology functors are rigid, namely, if  $\mathrm{D}_n(S/R;-)=0$ for some $n\ge 2$, then $\mathrm{D}_n(S/R;-)=0$ for all $n\ge 2$ and $\vf$ is locally complete intersection. This extends  Avramov's theorem that draws the same conclusion assuming $\mathrm{D}_n(S/R;-)$ vanishes for all $n\gg 0$, confirming a conjecture of Quillen. The rigidity of Andr\'e-Quillen functors is deduced from a more general result about the higher cotangent modules which answers a question raised by Avramov and Herzog, and subsumes a conjecture of Vasconcelos that was proved recently by the first author.  The new insight leading to these results concerns the equivariance of a map from Andr\'e-Quillen cohomology to Hochschild cohomology defined using the universal Atiyah class of $\varphi$.
\end{abstract}

\maketitle

\section*{Introduction}
\label{sec:intro}

In this work we discover new rigidity properties of the cotangent complex associated to a map of commutative noetherian rings, or of locally noetherian schemes. The phenomena we consider are all local, so we focus in the introduction on a map of commutative noetherian rings $\vf\colon R\to S$; its cotangent complex is denoted $\cotan{\vf}$. One  rigidity property can be stated using the Andr\'e-Quillen homology functors $\aqh nRS- =\HH n{\cotan{\vf}\lotimes S-}$ on the category of $S$-modules.

\begin{theorema}
\label{intro:thma}
Let $\vf\colon R\to S$ be a map of commutative noetherian rings, that is locally of finite flat dimension. If $\aqh nRS-=0$ for some $n\ge 1$, then $\vf$ is locally complete intersection.
\end{theorema}

Hitherto this result was known if $R$ contains $\mathbb Q$ as a subring, and  also when $\aqh nRS-=0$ for \emph{all} $n\gg 0$. Both results are due to Avramov~\cite{Avramov:1999}; for the former see also Halperin~\cite{Halperin:1987}. The hypothesis of the latter result is equivalent to the finiteness of the flat dimension of $\cotan{\vf}$, and in this form it settled a longstanding conjecture of Quillen~\cite{Quillen:1968}. In low degrees the vanishing of individual $\aqh nRS-$ has classically been used to characterise formally \'etale, formally smooth, and locally complete intersection maps, and each of these conditions imply that $\aqh nRS-=0$ for all $n\ge 2$; this was proved Andr\'e~\cite{Andre:1974} and Quillen~\cite{Quillen:1968} when $\vf$ is essentially of finite type, and in general by Avramov~\cite{Avramov:1999}.

The cotangent complex is well-known to be significantly more complicated  in positive and mixed characteristic than it is in characteristic zero. Avramov's proof of Quillen's conjecture used sophisticated positive characteristic methods to connect simplicial and differential graded invariants. No other proofs or simplifications of Avramov’s argument have appeared before this work. Our proof of Theorem \ref{intro:thma} is independent of \cites{Avramov:1999,Halperin:1987}, is significantly shorter, and pays no regard to the characteristic of the rings involved. Our arguments highlight a new feature controlling the behaviour of Andr\'e-Quillen cohomology, namely its torsion with respect to the action of Hochschild cohomology; see Theorem \ref{intro:thmc}.

We deduce Theorem \ref{intro:thma}  from a more general result on the \emph{higher cotangent modules} $\cotam n{\vf}$ introduced in \cite{Avramov/Herzog:1994}. Assuming that $\cotan{\vf}$ is represented by a bounded below complex of projectives, these are defined to be the $S$-modules
\[
\cotam n{\vf}\colonequals \HH n{\cotan{\vf}_{\ges n}} \qquad\text{for $n\ge 0$.}
\]
Each $\cotam n{\vf}$ is finitely generated when $\vf$ is essentially of finite type---that is to say, $S$ is a localisation of a finitely generated $R$-algebra. The cotangent complex is well-defined up to quasi-isomorphism, and the higher cotangent modules are correspondingly well-defined objects in the singularity category of $S$ in sense of Buchweitz and Orlov; see Section \ref{ch:geometry}.

The $\cotam n{\vf}$ can be thought of as higher analogues of the module $\cotam 0{\vf}= \Omega_{S/R}^1$ of K\"ahler differentials, and therefore our main result is a higher analogue of the Jacobian criterion:

\begin{theorema}
\label{intro:thmb}
Let $\vf\colon R\to S$ be a map of commutative noetherian rings, essentially of finite type and locally of finite flat dimension. If for some integer  $n\ge 1$ the $S$-module $\cotam n{\vf}$  has finite flat dimension, then $\vf$ is locally complete intersection.
\end{theorema}

This answers a question posed by Avramov and Herzog~\cite{Avramov/Herzog:1994}. It also subsumes Theorem \ref{intro:thma}, since if $\aqh nRS-=0$ then both $\cotam n{\vf}$ and  $\cotam {n-1}{\vf}$ are flat; this is explained in Section~\ref{se:cotangent-modules}. 

The missing case $n=0$ is related to a still-open conjecture of Eisenbud and Mazur \cite{Eisenbud/Mazur:1997}; see \cite{Briggs:2020a}*{3.4} for discussion. The case $n=1$ is a recent result of the first author~\cite{Briggs:2020a}, confirming a conjecture of Vasconcelos on the conormal module of a surjective homomorphism. Our arguments share an ingredient with \cite{Briggs:2020a}, but are for the most part completely different.  

The fundamental new input that makes our proof possible is a morphism
\[
\atiyahclass \vf \colon S \lra \susp \cotan{\vf}
\]
in the derived category of the enveloping algebra of $S$ over $R$.  This map was introduced by Illusie in \cite{Illusie:1971}, and called the  \emph{universal Atiyah class} of $\vf$ by Buchweitz and Flenner~ \cite{Buchweitz/Flenner:2008}. It has played an important role in applications of the cotangent complex to deformation theory, among other things. As far as we are aware, it has not been used before to study rigidity properties of the cotangent complex.

Here we focus on a new aspect of the universal Atiyah class: For any $S$-module $\ell$,  Yoneda composition with $\atiyahclass{\vf}$ yields a degree one map 
\[
\atiyah{\vf}{\ell}\colon \aqch {*}RS{\ell} \lra \hcm[*+1]RS{\ell}\,,
\]
from the Andr\'e-Quillen cohomology with coefficients in $\ell$, to the Hochschild cohomology with coefficients in $\ell$ viewed as a symmetric bimodule. 

Taking $\ell$ to be a residue field of $S$, it turns out that $\hcm RS{\ell}$ is a Hopf algebra over $\ell$.  The crucial result in this context is that the $\atiyah{\vf}{\ell}$ is equivariant with respect to the characteristic action of $\hcm  RS{\ell}$ on Andr\'e-Quillen cohomology, and the adjoint action of $\hcm RS{\ell}$ on itself, which arises from its Hopf algebra structure. For this reason we call $\atiyah{\vf}{\ell}$ the \emph{Atiyah character} of $\vf$. Seen through the looking glass~\cite{Avramov/Halperin:1983} it is analogous to the Hurewicz map from the rational homotopy of a loop space to its rational homology.

The equivariance statement is Theorem~\ref{th:atiyah-equivariance}. It is deduced from Theorem~\ref{th:equivariance}, which is a general statement about the interplay between symmetric bimodules and all bimodules.  Combining the equivariance theorem with an earlier result  of Avramov and Halperin~\cite{Avramov/Halperin:1987}  on the structure of $\hcm RS{\ell}$ leads to the next result concerning torsion in Andr\'e-Quillen cohomology with respect to the characteristic action of Hochschild cohomolgy.

\begin{theorema}
\label{intro:thmc}
Let $\vf\colon R\to S$ be a surjective local map of finite projective dimension, and $\ell$ the residue field of $S$. If $\hcm[\ges n]RS\ell\cdot  \aqch 1RS{\ell}=0$ for some integer $n$, then $\vf$ is complete intersection.
\end{theorema}

This exploits the fact that $\hcm RS{\ell}$ is the universal envelope of the homotopy Lie algebra of $\vf$, introduced by Avramov~\cite{Avramov:1984a}. Typically, the $\ell$-algebra $\hcm RS{\ell}$ is  not finitely generated, much less commutative, and it is  hard to verify that elements in $\aqch 1RS{\ell}$ are torsion. One scenario where it is easy to do so is when the natural  augmentation $\ve\colon \cotan{\vf}\to \susp  \aqch 1RS{\ell}$ factors through a perfect complex of $S$-modules. With this observation, Theorem \ref{intro:thmb} is a simple corollary of Theorem \ref{intro:thmc}; this is explained in Section \ref{se:cotangent-modules}. The condition that $\ve$ factors through a perfect complex is tantamount to the condition that it induces the zero map in the singularity category of $S$. This perspective leads to extensions of Theorems \ref{intro:thma} and \ref{intro:thmb} to locally noetherian schemes, recorded in \ref{ch:geometry}.

\begin{ack}
We  acknowledge with pleasure our  huge intellectual debt to Avramov and Halperin, and to Buchweitz.  The debt to Avramov is not only for his writings: The second author has had the benefit of innumerable conversations with him, over a period of more than twenty years,  on the mathematics surrounding the cotangent complex. And it was  Buchweitz who stressed the importance of the universal Atiyah class and its potential in commutative algebra. We thank Janina Letz, Linquan Ma, and two anonymous referees for comments and suggestions on an earlier version of this manuscript.
\end{ack}
 
\section{Graded Hopf algebras}
\label{se:Hopf}
In this section we collect some basic notions concerning graded Hopf algebras over fields.  Everything we need is already in the paper of Milnor and Moore~\cite{Milnor/Moore:1965}; see also a pre-publication of the same, reprinted in \cite{Milnor:2010}*{pp.~7}. Throughout this section we work over a field $\ell$ and  in the category of graded $\ell$-vector spaces; in particular tensor products and the module of homomorphisms are taken in this category.

\begin{chunk}
\label{ch:Hopf-algebras}
Let $(A,\mu,\eta,\Delta,\ve)$ be a graded Hopf algebra over $\ell$, with product $\mu$, unit $\eta$,  coproduct $\Delta$, and counit $\ve$. We only consider positively graded  (either upper or lower) and connected algebras: $A^0=\ell$. Every such Hopf algebra has an \emph{antipode} $\sigma\colon A\to A$, namely, an inverse to the identity map on $A$ for the convolution product on $\Hom{\ell}AA$; see \cite{Milnor/Moore:1965}*{\S8}, where this is called the ``conjugation". Being an inverse, the antipode is uniquely defined by that property.
\end{chunk}

Under suitable finiteness hypotheses the dual of a Hopf algebra is also a Hopf algebra. This is explained below.

\begin{chunk}
\label{ch:hopf-dual}
Let $(A,\mu,\eta,\Delta,\ve)$ be a graded Hopf algebra over $\ell$ of \emph{finite type}: the $\ell$-vector space $A_i$ is finite dimensional for each $i$.  Set $\dual A\colonequals \Hom{\ell}A{\ell}$,  the dual graded vector space. The finiteness hypothesis implies that the natural map 
\begin{align*}
\delta\colon & \dual A \otimes_{\ell} \dual A \lra \Hom{\ell}{A\otimes_{\ell}A}{\ell} \quad \text{where}\\
	& \delta(\alpha \otimes \beta)(a\otimes b) = (-1)^{|a||\beta|}\alpha(a)\beta(b)\,,
\end{align*}
is an isomorphism. It follows that $\dual A$ has a structure of a Hopf algebra over $\ell$, with product and coproduct defined by the compositions
\begin{align*}
&\dual A \otimes_{\ell} \dual A \xra{\ \delta\ } \Hom{\ell}{A\otimes_{\ell}A}{\ell} \xra{\ \dual\Delta\ } \Hom{\ell}A{\ell} = \dual A \\
&\dual A = \Hom{\ell}A{\ell} \xra{\ \dual{\mu}\ } \Hom{\ell}{A\otimes_{\ell}A}{\ell} \xra{\ \delta^{-1}\ } \dual A\otimes_{\ell} \dual A\,.
\end{align*}
If $\sigma$ is the  antipode on $A$ then  $\dual \sigma$ is the antipode on $\dual A$. 

We use the following observation: Fix $\alpha\in \dual A$ and say $(\delta^{-1}\dual{\mu})(\alpha) = \sum \alpha_{1}\otimes \alpha_{2}$ in $\dual A\otimes_{\ell}\dual A$. Then for any $a,b$ in $A$  one has
\begin{align}
\label{eq:dual-coproduct}
\begin{aligned}
\alpha(ab) 
	&= \dual{\mu}(\alpha)(a\otimes b) \\
	&= \sum \delta(\alpha_{1}\otimes \alpha_{2})(a\otimes b) \\ 
	&= \sum (-1)^{|a||\alpha_2|}\alpha_{1}(a) \alpha_{2}(b)  \,.
\end{aligned}
\end{align}
The reader may find reassurance in the last lines of \cite{Serre:1993}*{pp.~39}. 
\end{chunk}

\begin{chunk}
\label{ch:primitive}
An element $a\in A_{\ges 1}$ in a Hopf algebra $A$ is \emph{primitive} if its coproduct satisfies
\[
\Delta(a) = a \otimes 1 + 1 \otimes a\,.
\]
For such an element $a$ it follows from the definition of the antipode that $\sigma(a) = -a$. 
\end{chunk}

When $A$ is a Hopf algebra, any $A$-bimodule action can be intertwined into a single left action described below.

\begin{chunk}
\label{ch:classical-adjoint}
Let $A$ be a Hopf algebra over $\ell$ and $W$ an $A$-bimodule.  The \emph{adjoint action} of $A$ on $W$ is the (left) $A$-module structure on $W$ defined as follows: For any $a\in A$ and $w\in W$ one has
\[
\ad(a)(w)= \sum (-1)^{|a_{2}||w|} a_{1}  w \ \sigma(a_{2}) \quad\text{where $\Delta(a)=\sum a_{1}\otimes a_{2}$.}
\]
In these terms the condition that $\sigma$ is the antipode is precisely the condition that it is an $\ell$-linear map such  that $\ad(a)=0$ on $A^0$ whenever $|a|\ne 0$; that is to say, $A^0$ is in the socle of $A$ for the adjoint action. 

When $a$ is primitive, the adjoint action has a simple description:
\[
\ad(a)(w) = [a,w]\colonequals a w - (-1)^{|a||w|} wa\,.
\]
This is the (graded) commutator of $a$ and $w$. 
\end{chunk}

\section{Cohomology of commutative dg algebras}
\label{se:supplemented}
This section too is mostly a recollection of known results, now about the existence of Hopf algebra structures on the homology and cohomology of  algebras. Here we take as basic references the books of Avramov~\cite{Avramov:1998}, Felix, Halperin, and Thomas~\cite{Felix/Halperin/Thomas:2001}, and  Gulliksen and Levin~\cite{Gulliksen/Levin:1969}. In what follows we have to use the derived enveloping algebra. For concreteness, we use commutative dg algebras to model this, and other derived rings, but one could as well use simplicial models.

\begin{chunk}
\label{ch:dga}
By a differential graded (abbreviated to ``dg") algebra  we shall mean a non-negatively graded complex $F=\{F_i\}_{i\ges 0}$ over some commutative ground ring ($\mathbb{Z}$ is always a choice), equipped with a graded-commutative multiplication, bilinear with respect to the ground ring, so that the differential satisfies the Leibniz rule. In our applications the dg algebras that arise will  be strictly graded-commutative, and with divided powers, but these structures will not be used in our arguments.

The derived category of dg $F$-modules (which we usually speak of as $F$-modules) will be denoted $\dcat{F}$. We use suitable resolutions as needed, including semi-free dg modules, and semi-free dg algebras which we call \emph{models}; see \cites{Avramov:1998, Felix/Halperin/Thomas:2001}.
\end{chunk}

Throughout this section $F$ will be a dg algebra equipped with a surjective map $F\to \ell$, where $\ell$ is a field. We recall, from \cite{Gulliksen/Levin:1969}, the construction of Hopf algebra structures on $\Tor F{\ell}{\ell}$ and its dual $\Ext F{\ell}{\ell}$.

\begin{chunk}
\label{ch:hopf-tor}
For $F\to \ell$ as above, $\Tor{F}{\ell}{\ell}$ has a natural structure of a graded-commutative Hopf algebra over ${\ell}$. The product is induced by the map
\[
({\ell}\lotimes F {\ell})\lotimes {\ell} ({\ell}\lotimes F {\ell})  \lra {\ell}\lotimes F {\ell}
\]
obtained by multiplication on the respective factors, and following the Koszul sign rule.  In homology this induces the second map below:
\[
\hh({\ell}\lotimes F {\ell}) \otimes_{\ell} \hh({\ell}\lotimes F{\ell}) \xra{\ \cong\ } 
	\hh(({\ell}\lotimes F {\ell})\lotimes {\ell} ({\ell}\lotimes F {\ell})) \lra \hh({\ell}\lotimes F{\ell})\,.
\]
The first one is the K\"unneth map, which is an isomorphism since ${\ell}$ is a field. Thus one has a map
\begin{equation}
\label{eq:hha-product}
\Tor F{\ell}{\ell}\otimes_\ell \Tor F{\ell}{\ell} \lra \Tor F{\ell}{\ell}
\end{equation}
that makes $\Tor F{\ell}{\ell}$ is a graded-commutative $\ell$-algebra. This is where the commutativity comes in play. 

The coproduct on $\Tor F{\ell}{\ell}$ is induced by the natural map
\[
{\ell}\lotimes F{\ell} \lra {\ell}\lotimes F {\ell} \lotimes F {\ell} 
	\simeq ({\ell}\lotimes F {\ell})\lotimes {\ell} ({\ell}\lotimes F {\ell})
\]
induced by the assignment $x\otimes y \mapsto x\otimes 1\otimes y$. Using the fact that the K\"unneth map is an isomorphism, this induces the map
\begin{equation}
\label{eq:coproduct-tor}
\Tor F{\ell}{\ell} \lra \Tor F{\ell}{\ell} \otimes_\ell \Tor F{\ell}{\ell}\,.
\end{equation}
It is easy to verify that this map is a morphism of $\ell$-algebras, and that it defines a coproduct on $\Tor F{\ell}{\ell}$. 

The antipode on $\Tor F{\ell}{\ell}$  is induced by the twisting map
\[
{\ell} \lotimes F{\ell} \lra {\ell}\lotimes F{\ell} \quad\text{where $a\otimes b\mapsto (-1)^{|a||b|}b\otimes a$.}
\]
The computations that show that with the structures defined above,  $\Tor F{\ell}{\ell}$ is a Hopf algebra over ${\ell}$ are straightforward; see \cite{Gulliksen/Levin:1969} that deals with the case where $F$ is a ring. 
The same arguments carry over to our context as well.
\end{chunk}

\begin{chunk}
\label{ch:hopf-ext}
In this paragraph we assume  that the dg algebra is  $F$ \emph{degreewise noetherian} by which we mean that the ring $F_0$ is noetherian, and the $F_0$-modules $F_i$ are finitely generated for $i\ge 1$; recall that $F_i=0$ for $i<0$.  These conditions imply that $\rank_{\ell}\Tor [i]F\ell\ell$ is finite for each $i$, and zero for $i<0$. The adjunction isomorphism
\[
\Hom {\ell}{{\ell}\lotimes F{\ell}}{\ell} \cong \RHom F{\ell}{\ell}
\]
induces an isomorphism of ${\ell}$-vector spaces
\[
\Ext F{\ell}{\ell} \cong \Hom {\ell}{\Tor F{\ell}{\ell}}\ell\,.
\]
Thus, as per the discussion in \ref{ch:hopf-dual}, the Hopf algebra structure on $\Tor F{\ell}{\ell}$ induces such a structure on $\Ext F{\ell}{\ell}$, and since the former is commutative, and latter is cocommutative.  The multiplication thus defined on $\Ext F{\ell}{\ell}$ is the usual one given by composition.

\end{chunk}

\section{Hochschild cohomology}
\label{se:hha-and-hca}
In this section we introduce a Hopf algebra structure on the Hochschild cohomology of commutative algebras with field coefficients, and connect this with the  Hopf algebra $\Ext F{\ell}{\ell}$  defined in the previous section, for a suitable $F$. Both descriptions will be crucial: we need the Hochschild approach in Sections \ref{se:products-general} and \ref{se:atiyah}, and the Yoneda approach is Section \ref{se:lci}.

Throughout $\vf\colon R\to S$ will be a morphism of dg algebras, as in \ref{ch:dga}.  The main case of interest is when $R$ and $S$ are rings, but it is convenient to work in the greater generality, for it allows one to replace rings with semi-free dg models when needed. 

\begin{chunk}
\label{ch:bimodules}
We write $\env RS$ for the derived enveloping algebra, $S\lotimes RS$, of  $S$ over $R$.  It is convenient to replace the $R$-algebra $S$ by a suitable model and assume $\env RS= S\otimes_RS$.  We consider the  morphisms of dg algebras:
\[
 \mu\colon \env RS \to S\quad\text{and} \quad \tau \colon \env RS \to \env RS
\]
where $\mu$ is the multiplication map and $\tau$ is the twisting map, defined on pure tensors by $s\otimes t \mapsto (-1)^{|s||t|} t\otimes s$.

An $\env RS$-module is the same as an $S$-bimodule on which the $R$-action is symmetric, and it will be useful to keep both perspectives in mind. Given $\env RS$-modules $M$ and $N$, the tensor product $M\otimes_SN$ is again an $\env RS$-module, with left and right actions inherited from the left and right factors, respectively: 
\[
s \cdot (m\otimes n)\colonequals sm \otimes n \qquad\text{and}\qquad (m\otimes n)\cdot s \colonequals m \otimes ns
\]
for $s$ in $S$ and $m\otimes n$ in $M\otimes_SN$. Viewing $S$ viewed as a bimodule via $\mu$, the natural maps $M\otimes_S S\to M$ and $S\otimes_SN \to N$ are isomorphisms of bimodules. We repeatedly exploit the fact that this induces a structure of a tensor-triangulated category on $\dcat {\env RS}$,  with product $-\lotimes S-$ and unit $S$, so that 
\begin{equation}
\label{eq:unit}
S\lotimes S M  \simeq M \quad\text{and}\quad M \lotimes S S\simeq M
\end{equation}
for each $\env RS$-module $M$. This tensor product is not symmetric, and this leads to interesting structures on cohomology, as will become clear soon.
\end{chunk}

\begin{chunk}
\label{ch:hh-defn}
Given an $\env RS$-module $M$, let  $\hcm RSM$ be the \emph{Hochschild cohomology} of the $R$-algebra $S$, with coefficients in $M$; thus
\[
\hcm RSM \colonequals \Ext [*]{\env RS}SM\,.
\]
This is also called the Shukla cohomology, the derived Hochschild cohomology, or the Hochschild-Quillen cohomology~\cite{Quillen:1968}. 
\end{chunk}

\begin{chunk}
\label{ch:fieldcoef}
Let $\vf\colon R\to S$ be as above.  In the remainder of the section we fix a surjective map $S\to \ell$, where $\ell$ is a field. In addition, we ask that $R$ and $S$ be degreewise noetherian, in the sense of \ref{ch:hopf-ext}, and that the map $\vf_0\colon R_0\to S_0$ be \emph{essentially of finite type}; that is to say, $S_0$ is a localisation of a finitely generated $R_0$-algebra. This implies that the dg algebra $\env RS$ can be chosen degreewise noetherian as well. In this context we construct operations making 
\[
\hcm RS\ell
\]
into a graded, cocommutative Hopf $\ell$-algebra.

 We fix a semi-free dg algebra model $\ve\colon \wt S\xra{\sim} S$ over $\env RS$  and compute Hochschild cohomology of an $S$-bimodule $M$ using the complex
\[
\Hom{\env RS}{\wt S}{M}\,.
\]
By our hypotheses of $R$ and $S$, we can and will choose $\wt S$ to be degreewise noetherian.
\end{chunk}

\begin{chunk}
\label{ch:hh-product}
 The product on $\hcm RS\ell$ uses only the $S$-bimodule structure of $\wt S$. The tensor product $\wt S\otimes_S \wt S$ is viewed as an $S$-bimodule in the usual way; see \ref{ch:bimodules}. It is semi-free over $\env RS$ since $\wt S$ is assumed so. There are two natural quasi-isomorphisms $ \wt S\otimes_S \wt S\to \wt S$, namely $1\otimes_S\ve$ and $\ve\otimes_S 1$, and both represent the same morphism in $\dcat{\env RS}$ since they are coequalised by the quasi-isomorphism $\ve$. This yields the quasi-isomorphism below:
\begin{align*}
\Hom{\env RS}{\wt S}{\ell} \otimes_{\ell} &\Hom{\env RS}{\wt S}{\ell} 
	\longrightarrow \Hom{\env RS}{\wt S\otimes_S \wt S}{\ell}
	\xleftarrow{\sim} \Hom{\env RS}{\wt S}{\ell}\\
    & f_1  \otimes f_2   \quad \mapsto\quad (\wt S\otimes_S \wt S \xrightarrow{f_1\otimes_S f_2} \ell\otimes_S\ell = \ell)\,.
\end{align*}
In cohomology, the composite map, which entails inverting the one on the right, defines the cup product $f_1\smile f_2$; see also the discussion in \ref{ch:cupproduct}, and also \cite{Buchweitz/Roberts:2015}*{\S4}. The unit is the augmentation $S\to \ell$ composed with $\ve$.
\end{chunk}

\begin{chunk}
\label{ch:hh-coproduct}
The coproduct on $\hcm RS{\ell}$ uses the multiplication map $\wt\mu\colon \wt S\otimes_{\env RS} \wt S \to \wt S$. In contrast to the bimodule tensor product used to define the cup product,  the tensor product here treats both factors $\wt S$ as $\env RS$-modules in the \emph{same} way, exploiting commutativity. This induces the  map on the left below:
\[
\wt\Delta\colon 
\Hom{\env RS}{\wt S}{\ell} \xra{\Hom{}{\wt\mu}{\ell}}\Hom{\env RS}{\wt S\otimes_{\env RS} \wt S}{\ell} 
	\xleftarrow{\cong} \Hom{\env RS}{\wt S}{\ell} \otimes_{\ell} \Hom{\env RS}{\wt S}{\ell}\,. 
\]
The isomorphism is standard; it exists because $\wt S$ is degreewise noetherian.  Passing to cohomology, and using the K\"unneth isomorphism, yields the coproduct on $\hcm RS{\ell}$. In fact $\wt\Delta$ makes $\Hom{\env RS}{\wt S}{\ell}$ into a dg coalgebra over $\ell$, with counit the dual of the structure map $\env RS\to \wt S$.
\end{chunk}

\begin{chunk}
\label{ch:hh-antipode}
The antipode on $\hcm RS{\ell}$ is defined using the twisting map $\tau$ on $\env RS$. Since $\wt S$ is semi-free over $\env RS$ and $\ve$ is a quasi-isomorphism of dg algebras, one can construct a commutative diagram 
\[
\begin{tikzcd}[row sep= 5mm]
\wt S \ar[d,"\ve" swap]  \ar[r,"{\wt\tau}"]  & \tau_*(\wt S) \ar[d,"\ve"] \\
S \ar[r, equal] & S
\end{tikzcd}
\]
of dg $\env RS$-modules, where $\tau_*$ is the restriction of scalars along the twisting map $\tau \colon \env RS \to \env RS$. This defines the antipode $\sigma\colonequals \Ext{\tau}{\wt\tau}{\ell}$. Thus, given  an $\env RS$-linear chain map $f\colon \wt S\to \ell$ its antipode $\sigma(f)$ is the composition 
\[
\wt S\xrightarrow{\ \wt \tau\,}\tau_*(\wt S) \xrightarrow{\tau_*(f)} \tau_*(\ell)=\ell\,.
\]
\end{chunk}

Here is the result announced at the beginning of the section.

\begin{theorem}
\label{th:hh=hopf}
With $R\to S\to\ell$ as in \emph{\ref{ch:fieldcoef}}, the operations described above endow $\hcm RS{\ell}$ with the  structure of a graded, cocommutative, Hopf $\ell$-algebra.
\end{theorem}

\begin{proof}
The associativity and unitality of the cup product---as well as the coassociativity and counitality of the coproduct---can be verified in a straightforward way. The graded cocommutativity of $\hcm RS{\ell}$ follows from the graded commutativity of $\wt S$. For the bialgebra identity, take chain maps $f$ and $g$ in $\Hom{\env RS}{\wt S}{\ell}$, and consider the following commutative diagram
\[
\begin{tikzcd}[row sep=5mm]
(\wt S\otimes_S\wt S)\otimes_{\env RS}(\wt S\otimes_S\wt S) \ar[d,"(\ve\otimes 1)\otimes (\ve \otimes 1)"',"\sim"] \ar[r] &  \wt S\otimes_S\wt S \ar[d,"\ve\otimes 1"',"\sim"]\ar[r,"f\otimes_Sg"]& \ell\\
\wt S\otimes_{\env RS}\wt S \ar[r] & \wt S\,,
\end{tikzcd}
\]
where the top row is given by $(x\otimes y) \otimes (u\otimes v)\mapsto (-1)^{|y||x|}xu\otimes yv$. The upper path represents $(\smile \otimes \smile)(1\otimes \tau\otimes 1) (\Delta(f)\otimes \Delta(g))$, while the lower path (inverting the quasi-isomorphism) represents $\Delta(f\smile g)$, so the diagram is witness to the bialgebra identity.

It remains to show that $\sigma$ is an antipode; the argument for this anticipates the proof of Theorem \ref{th:equivariance}; confer also the discussion in \ref{ch:classical-adjoint}. Let $f$ be a chain map in $\Hom{\env RS}{\wt S}{\ell}$ with $\Delta(f) =\sum f_1\otimes f_2$. We must show that 
\[
{\textstyle \sum}  f_1\smile \sigma(f_2) = \eta\ve(f) = {\textstyle \sum} \sigma(f_1)\smile f_2 \,.
\]
We verify the equality on the left; the argument for the one on the right is similar. It is sufficient to verify that $\sum f_1\smile \sigma(f_2)=0$ when $f$ is in the kernel of the counit $\ve$. Denote by $\iota\colon S\to \env RS$ the inclusion into the left tensor factor, and consider the commutative diagram
\[
\begin{tikzcd}[column sep=20mm,row sep=6mm]
\wt S\otimes_S \wt S \ar[d,"\kappa"'] \ar[r,"\sum f_1\smile \sigma(f_2) "] & \ell \ar[d, equals] \\
\mu_* {\iota}_*(\wt S)\ar[r,"\mu_* {\iota}_*(f)"] & \mu_* {\iota}_*(\ell)\,.
\end{tikzcd}
\]
Here $\kappa( x\otimes y) =x\wt\tau(y)$; one checks directly that $\kappa$ is well-defined and $\env RS$-linear. Finally, if $\ve(f)=0$ then $\mu_* {\iota}_*(f)=0$, since already the class of ${\iota}_*(f)$ vanishes in 
\[
\Hom{S}{{\iota}_*(\wt S)}{\ell}\simeq \Hom{S}{S}{\ell} = \ell\,. \qedhere
\]
\end{proof}

Next, we identify the Hochschild Hopf algebra just constructed with a Yoneda Hopf algebra of the previous section.

\begin{chunk}
\label{ch:fiber}
We remain in the context of \ref{ch:fieldcoef}, and set $F\colonequals {\ell} \lotimes R S $, viewed as a dg $\ell$-algebra. The dg algebra $\env RS$ can be chosen degreewise noetherian, so the same property is inherited by  $F$ and its homology algebra.

The map $S\to \ell$ induces the morphism 
\[
\env RS = S\lotimes RS \to {\ell} \lotimes RS = F
\]
of dg algebras. Moreover the composition  $\env RS\xrightarrow{\mu}S\to {\ell}$ factors through this map and induces quasi-isomorphisms
\[
F\lotimes {\env RS} S = (\ell\lotimes RS) \lotimes {\env RS} S \simeq \ell\lotimes SS \simeq {\ell}\,.
\]
where the first one is the standard diagonal isomorphism.  Apply $\RHom{F}-{\ell}$ and using adjunction  yields an isomorphism
\[
\RHom{\env RS}S{\ell}\simeq  \RHom F{\ell}{\ell}\,.
\]
In summary, there is an isomorphism of $\ell$-vector spaces
\begin{equation}
\label{eq:hh=ext}
\hcm RS{\ell}\colonequals \Ext[*]{\env RS}S{\ell} \cong \Ext [*]F{\ell}{\ell}\,.
\end{equation}
Here is an explicit description in terms of morphisms: Given $\zeta\colon S\to \susp^n\ell$ in $\dcat{\env RS}$ the corresponding morphism in $\dcat F$ is the composition
\[
\ell \cong \ell \lotimes S S \xra{\ \ell\lotimes S \susp^n \zeta} \susp^n \ell\lotimes S \ell \lra \susp^n \ell
\]
where the map on the right is induced by the multiplication on $\ell$. On the other hand, given $\xi\colon \ell\to \susp^n\ell$ in $\dcat F$ composing with the map $S\to \ell$ gives the morphism 
\[
S\lra \ell \xra{\ \xi\ } \susp^n\ell
\]
in $\dcat {\env RS}$.  A straightforward computation shows that these assignments are inverse isomorphisms yielding \eqref{eq:hh=ext}.
\end{chunk}

\begin{chunk}
Since $F$ is degreewise noetherian, from the discussion in \ref{ch:hopf-ext} there is  a natural structure of a Hopf $\ell$-algebra on $\Ext[*]{F}{\ell}{\ell}$. On the other hand, $\hcm RS{\ell}$ is also a Hopf $\ell$-algebra, by Theorem~\ref{th:hh=hopf}. The result below should come as no surprise. 

\begin{proposition}
\label{pr:HH=Ext}
The isomorphism $\hcm RS{\ell}\cong \Ext[*]F{\ell}{\ell}$ in \eqref{eq:hh=ext} is compatible with the Hopf algebra structures.
\end{proposition}

\begin{proof}
We only need to check compatibility with product and coproduct, for the anitpode is determined by those structures. 

We can assume that $S$ is semi-free over $R$ and so identify $F$ with the dg $\ell$-algebra $\ell\otimes_RS$.
Let $\ve\colon \wt S\to S$ be a semi-free dg model of $S$ over $\env RS$. Since $\wt S$ is also semi-free over $S$, applying $\ell\otimes_S(-)$ to $\ve$ yields a quasi-isomorphism
\[
X\colonequals \ell \otimes_S \wt S \xra{ \ve' } \ell \,.
\]
of dg algebras over $F$. Note that $X$ is semi-free  over $F$.

Applying $\ell\otimes_S(-)$ to the  multiplication map on $\wt S$ yields the multiplication map on $X$. The coproducts on $\hcm RS{\ell}$ and $\Ext[*]F{\ell}{\ell}$ are induced by these maps---see the discussion in \ref{ch:hopf-tor}, \ref{ch:hopf-ext}, and \ref{ch:hh-coproduct}---so we deduce that they coincide. 

As to the product structures, the one on $\hcm RS{\ell}$ is induced by the quasi-isomorphism of dg $\env RS$-algebras
\[
\wt S \xleftarrow{\ 1\otimes \ve\ } \wt S\otimes_S \wt S\,,
\]
as explained in \ref{ch:hh-product}. This induces the top row of the following  diagram 
\[
\begin{tikzcd}[column sep=large]
X\otimes_F X\cong (X\otimes_\ell \ell) \otimes_F X \ar[dr] \ar[r,leftarrow,"\simeq" swap,  "1\otimes \ve'\otimes 1"] 
    & (X\otimes_\ell X) \otimes_F X \ar[d] \\ 
    & X\otimes_F X\otimes_F X
\end{tikzcd}
\]
In the top row, the action of $F$ on $X\otimes_\ell\ell$ is through $X$. The diagonal map is defined by the assignment $x\otimes y\mapsto x\otimes 1\otimes y$. The diagram commutes in the derived category of dg $\ell$--modules, for the two maps in question are coequalized by the quasi-isomorphism $1\otimes \ve'\otimes 1$. It remains to observe that the diagonal map induces the product on $\Ext[*]{F}{\ell}{\ell}$; see \ref{ch:hopf-tor} and \ref{ch:hopf-ext}.
\end{proof}
\end{chunk}

\section{Actions of Hochschild cohomology}
\label{se:products-general}
This section concerns the action of $\hcm RS{\ell}$ on certain cohomology modules arising from bimodules.  The main new result  is Theorem~\ref{th:equivariance} and that is at the heart of all that follows. We prepare the ground to state and prove it.

\begin{chunk}
\label{ch:hypotheses}
Throughout  $\vf\colon R\to S$ and $S\to \ell$ are morphisms of dg algebras such that:
\begin{enumerate}[\quad\rm(1)]
\item
$R_0$ is a noetherian ring and the $R_0$-module $R_i$ is finitely generated for $i\ge 1$;
\item
$S_0$ is essentially of finite type as an $R_0$-algebra, and the  $S_0$-module $S_i$ is finitely generated for $i\ge 1$;
\item
$S\to \ell$ is a surjective map with $\ell$ a field. 
\end{enumerate}
These are hypothesis from \ref{ch:fieldcoef} onward, and we freely draw upon the discussion in the preceding section.
\end{chunk}

\begin{chunk}
\label{ch:cupproduct}
 Given $\env RS$-modules $M$ and $N$ one gets an external product
\begin{equation}
\label{eq:cap-product}
\begin{tikzcd}
\Ext {\env RS}M{\ell} \otimes_{\ell} \Ext {\env RS}N{\ell} \ar[r,"\pitchfork"] 
	 &\Ext {\env RS}{M\lotimes SN}{\ell}\,,
\end{tikzcd}
\end{equation}
where $M\lotimes SN$ is viewed as an $\env RS$-module as explained in \ref{ch:bimodules}. Associativity of tensor products implies that this product is associative in the obvious sense. Setting $M=S=N$ one gets that $\Ext{\env RS}S{\ell}$, that is to say, $\hcm RS{\ell}$ is a graded $\ell$-algebra. This algebra structure is the same as the one introduced in \ref{ch:hh-product}.  Moreover,  given \eqref{eq:unit}, specialising $M$ or $N$ to $S$, yields the following:

\begin{proposition}
\label{pr:cap-product}
For each $M$ in $\dcat{\env RS}$, the product defined via \eqref{eq:cap-product} endows $\Ext{\env RS}M{\ell}$ with a natural structure of a $\hcm RS{\ell}$-bimodule. \qed
\end{proposition}
\end{chunk}

One can describe this explicitly on models; this will be useful later on.

\begin{chunk}
\label{ch:bimod}
We work with models as in \ref{ch:fieldcoef}. Let $M$ be a semi-free $\env RS$-module. The map
\[
\wt S\otimes_S M \otimes_S \wt S \lra M \quad \text{where $s\otimes m\otimes s'\mapsto sms'$}
\]
is evidently $\env RS$-linear, where the bimodule structure on the left is the natural one,  via the outer factors of the tensor product. It is also a quasi-isomorphism, by \eqref{eq:unit}.  Given chain maps $\alpha,\beta$ in $\Hom{\env RS}{\wt S}{\ell}$ and $f$ in $\Hom{\env RS}M{\ell}$, the class of  $\alpha\cdot f\cdot \beta$ is represented by the commutative diagram
\[
\begin{tikzcd}
 M \ar[d,"\alpha\cdot f\cdot \beta",swap]  \ar[r,leftarrow, "\sim"] & \wt S \otimes_S M \otimes_S \wt S \ar[d,"\alpha\otimes f\otimes \beta"] \\
 \ell \ar[r,leftarrow] & \ell\otimes_S \ell\otimes_S \ell 
\end{tikzcd}
\]
in $\dcat{\env RS}$, where the arrow in the bottom is given by the multiplication on $\ell$.
\end{chunk}

Recall from Theorem~\ref{th:hh=hopf}  that $\hcm RS{\ell}$ is a Hopf algebra. Its two-sided action on $\Ext{\env RS}M{\ell}$ can thus be combined into a single adjoint action, as explained in \ref{ch:classical-adjoint}. Here is a chain-level description.

\begin{chunk}
\label{ch:hh-adjoint}
Let $\alpha$ in $\Hom{\env RS}{\wt S}{\ell}$ be a chain map and using \ref{ch:hh-coproduct} write
\[
\wt\Delta(\alpha)= {\textstyle \sum}  \alpha_1\otimes \alpha_2\,.
\]
 A caveat: The maps $\alpha_1$ and $\alpha_2$ need not be chain maps, though they can be chosen to be so,  at the expense of replacing equality above by an equality of cohomology classes. Observe that, by definition, one has
\begin{equation}
\label{eq:hh-coproduct}
\alpha(xy) = {\textstyle \sum}  (-1)^{|\alpha_2||y|} \alpha_{1}(x)\alpha_{2}(y) \quad\text{for all $x,y\in \wt S$.}
\end{equation}
This is just the chain-level version of \eqref{eq:dual-coproduct}.

Let $f\colon M\to \ell$ be a chain map.  It follows from the description of the coproduct~\ref{ch:hh-coproduct} and the antipode~ \ref{ch:hh-antipode} that the adjoint action of the class of $\alpha$ on the class of $f$ is represented by the following commutative diagram
\begin{equation}
\label{eq:hh-adjoint}
\begin{tikzcd}
 M \ar[d,"\ad(\alpha)\cdot f" swap]  \ar[r,leftarrow, "\sim"] & \wt S \otimes_S M \otimes_S \wt S \ar[d,"\sum (-1)^{|\alpha_2| |f|} \alpha_1 \otimes f\otimes (\alpha_2\wt\tau)"] \\
 \ell \ar[r,leftarrow] & \ell\otimes_S \ell\otimes_S \ell 
\end{tikzcd}
\end{equation}
 in $\dcat{\env RS}$. It can be checked directly that the map on the right is a chain map.
\end{chunk}

We need one more piece of structure: For any $S$-module $M$, the Hochschild cohomology algebra $\hcm RS{\ell}$ acts on $\Ext SM{\ell}$ through Yoneda composition. A chain level description of this  \emph{characteristic action}  is given below.

\begin{chunk}
\label{ch:Yoneda}
Let $M$ be a semi-free $S$-module and consider the $S$-module $\wt S\otimes_S M$, where $S$ acts through the left-hand factor. The map  $\wt S\otimes_S M\lra  M$, where $s\otimes m\mapsto sm$,  is evidently $S$-linear; it is also a quasi-isomorphism, by \eqref{eq:unit}. Given chain maps $\alpha$ in $\Hom{\env RS}{\wt S}{\ell}$ and $f$ in $\Hom SM{\ell}$, the class of $\alpha\cdot f$ is represented by 
\[
\begin{tikzcd}
 M \ar[d,"\alpha\cdot f",swap]  \ar[r,leftarrow, "\sim"] & \wt S \otimes_S M  \ar[d,"\alpha\otimes f"] \\
 \ell \ar[r,leftarrow] & \ell\otimes_S \ell 
\end{tikzcd}
\]
where the arrow in the bottom row is given by multiplication on $\ell$.
\end{chunk}

The result below is the one we have been working towards. It is about the restriction functor $\mu_*\colon\dcat{S}\to \dcat{\env RS}$ associated to the multiplication map $\mu\colon \env RS\to S$.  

\begin{theorem}
\label{th:equivariance}
Let $R\to S\to \ell$ be as in \ref{ch:hypotheses}.  For each $S$-module $M$ the  map 
\[
\mu_*\colon \Ext SM\ell\longrightarrow \Ext{\env RS}{\mu_*(M)}\ell
\]
is linear with respect to the characteristic action of $\hcm RS{\ell}$ on the left and its adjoint action on the right.
\end{theorem}

\begin{proof}
The proof becomes a direct computation, once we recall the actions involved in a convenient form. We can assume we are in the context of  \ref{ch:fieldcoef}, and also that the $S$-module $M$ is semi-free. Thus elements in $\hcm RS{\ell}$ and $\Ext SM{\ell}$ are represented by chain maps in $\Hom{\env RS}{\wt S}{\ell}$ and $\Hom SM{\ell}$, respectively.

Since $M$ is semi-free over $S$, the  modules
\[
\wt S\otimes_S \mu_*(M) \otimes_S \wt S  \quad\text{and} \quad \wt S\otimes_S M
\]
are semi-free over $\env RS$ and $S$, respectively. As usual, the $S$-bimodule structure on the module on the left is via the outer factors of the tensor product. Both these modules are quasi-isomorphic to $M$, as explained in  \ref{ch:bimod} and \ref{ch:Yoneda}. It can be checked directly that the map
\begin{align*}
\kappa\colon & \wt S\otimes_S \mu_*(M) \otimes_S \wt S \lra \mu_* (\wt S\otimes_S M)  	\\
		      & x\otimes m\otimes y\mapsto (-1)^{|m||y|}x\wt\tau(y)\otimes m
\end{align*}
is well-defined, $\env RS$-linear, and a lifting of the identity on $M$ across $\mu$. Thus the map $\mu_*$ in the statement of the theorem is induced by the map
\[
 \Hom S{\wt S\otimes_SM}{\ell} 	\xra{\Hom{\mu}{\kappa}{\ell}} \Hom {\env RS}{\wt S\otimes_S\mu_*(M)\otimes_S \wt S}{\ell}\,.
\]
The action of $\hcm RS{\ell}$ on the source and target of this map have been described in \ref{ch:Yoneda} and \ref{ch:bimod}, respectively.

Let $\alpha$ in $\Hom{\env RS}{\wt S}{\ell}$ be a chain map, and write $\wt\Delta(\alpha)= \sum \alpha_{1}\otimes \alpha_{2}$ as in \ref{ch:hh-coproduct}. Let $f$ in $\Hom SM{\ell}$ be a chain map, set 
\[
g\colonequals \sum (-1)^{|\alpha_2||f|}\alpha_1 \otimes \mu_*(f) \otimes (\alpha_2\wt\tau)\colon 
	\wt S\otimes_S \mu_*(M) \otimes_S \wt S\lra \ell \otimes_S \ell\otimes_S\ell\,,
\]
and consider the following diagram of $\env RS$-modules:
\[
\begin{tikzcd}[column sep=small]
\mu_*(M)  \ar[d, "\ad(\alpha)\cdot \mu_*(f)" swap] \ar[r,leftarrow,"\sim"] 
	& \wt S\otimes_S \mu_*(M) \otimes_S \wt S \ar[d, "g"] \ar[rr,"\kappa"]
		 && \mu_*(\wt S\otimes_S M)  \ar[d, "\alpha\otimes f"] \ar[r,"\sim"] & M \ar[d,"\alpha\cdot f"]\\
\ell \ar[r,leftarrow] &\ell \otimes_S \ell \otimes_S \ell \ar[r] & \ell &\ell\otimes_{S} \ell \ar[l] \ar[r] & \ell
\end{tikzcd}
\]
where the squares on the left and on the right commute, by  \ref{ch:hh-adjoint} and \ref{ch:Yoneda} respectively.  We claim that the diagram in the middle is also commutative: going down from the top of the rectangle along the left edge is the map
\[
 x\otimes m\otimes y \mapsto  (-1)^{|f||x|} \alpha(x\wt\tau(y))f(m) \,.
\]
This computation uses  \eqref{eq:hh-coproduct}, and the fact that $f(m)=0$ unless $|f|=-|m|$. We leave it to the reader's pleasure to verify that going the other way gives the same map, with the correct sign. This justifies the equivariance of $\mu_*$.
\end{proof}

\section{The Atiyah character}
\label{se:atiyah}
In this section $\vf\colon R\to S$ is a map of commutative noetherian rings and  $\cotan{\vf}$ is its cotangent complex. We sometimes write $\Ext [*]S{\cotan \vf}N$ for Andr\'e-Quillen cohomology $\aqch{*}SR N$ as it makes certain constructions  transparent. We begin by sketching the construction of the cotangent complex; for details see \cites{Quillen:1968,Illusie:1971,Andre:1974}. 

\begin{chunk}
\label{ch:cotangent}
As usual $\env RS$ is the derived enveloping algebra of $S$ over $R$. Following \cite{Quillen:1968}*{\S6},  let $R\to P\xrightarrow{\sim} S$ be a simplicial resolution of $S$, so that $\env RS=S\otimes_RP$ is a simplicial model for the derived enveloping algebra. Set
\[
J\colonequals \Ker(S\otimes_RP\to S)
\]
so $J/J^2$ is the cotangent complex of $\vf$,  by  \cite{Quillen:1968}*{\S 6}. Thus there is an exact sequence
\begin{equation}
\label{eq:principalseq}
0\lra \cotan \vf \lra (S\otimes_R P)/J^2\lra S\lra 0
\end{equation}
of complexes over $S\otimes_R P=\env RS$. The corresponding connecting morphism
\[
\atiyahclass \vf\colon S\lra \susp \cotan{\vf}
\]
 in $\dcat{\env RS}$ is the \emph{universal Atiyah class} of $\vf$.  This was map was defined by Illusie \cite{Illusie:1971}*{IV.2.3.6.2}, whilst the terminology
 is borrowed from \cite{Buchweitz/Flenner:2008}.
\end{chunk}

\begin{chunk}
\label{ch:atchar}
Fixing an $S$-module $N$, one can dualise  the universal Atiyah class, with respect to $N$, along the multiplication map $\mu\colon \env RS\to S$ to obtain 
\[
\atiyah{\vf}{N}\colon \Ext[*] S{\cotan \vf}N\lra \hcm[*+1] RSN\,.
\]
This is a map of degree one from the Andr\'e-Quillen cohomology to the Hochschild cohomology of $S$ over $R$, both with coefficients in $N$.  We call this map the \emph{Atiyah character} of $\vf$ with coefficients in $N$, mainly because of Theorem~\ref{th:atiyah-equivariance}.
By definition $\atiyah[n]{\vf}{N}$, the component in degree $n$, factors as
\begin{equation}
\label{eq:atiyahdef}
\Ext[n]S{\cotan{\vf}}{N} \xrightarrow{\mu_*} \Ext[n]{\env RS}{\cotan{\vf}}{N} \xrightarrow{\dual{(\atiyahclass \vf)}}\Ext[n+1]{\env RS}{S}{N}\,;
\end{equation}
where $\dual{(\atiyahclass {\vf})}$ is used as a shorthand for $\Ext {\env RS}{\atiyahclass \vf}N$. 

Specialising $N$ to $\ell$ and applying Proposition~\ref{pr:cap-product} and Theorem~\ref{th:equivariance} yields:

\begin{theorem}
\label{th:atiyah-equivariance}
Let  $R$ be a commutative noetherian ring, $\vf\colon R\to S$ a map essentially of finite type, and $S\to \ell$ a surjection onto a field.  The Atiyah character
\[
\atiyah {\vf}{\ell}\colon \aqch{*}RS {\ell} \lra \hcm[*+1] RS{\ell}
\]
 with coefficients in $\ell$,  is equivariant with respect to the characteristic action of $\hcm RS{\ell}$ on the left and the adjoint action on the right. \qed
\end{theorem}
\end{chunk}

\begin{chunk}
Recall that $\HH 0{\cotan{\vf}}$ is the module of K\"ahler differentials of $S$ over $R$.  Thus one can identify $\Ext[0]S{\cotan{\vf}} N$ with $\mathrm{Der}_R(S,N)$, the module of $R$-linear derivations from $S$ into $N$. The Atiyah character in degree zero  is a familiar isomorphism
\[
\atiyah[0]{\vf}N\colon \mathrm{Der}_R(S,N)\xra{\ \cong\ } \hcm[1]RS{N}\,.
\]

When $\vf$ is surjective, with kernel $I$, one has
\begin{equation}
\label{eq:cotan01}
\HH 0{\cotan{\vf}} =0 \quad\text{and}\quad \HH 1{\cotan{\vf}} \cong I/I^2\,.
\end{equation}
Thus one gets an isomorphism
\[
\Ext [1]S{\cotan{\vf}}N\cong \Hom S{\HH 1{\cotan{\vf}}} N \cong \Hom S{I/I^2}N\,.
\]
The degree one part of the Atiyah character will play a crucial role in the sequel.

\begin{lemma}
\label{le:at2}
When $\vf$ is surjective, the Atiyah character is bijective in degree one: 
\[
\atiyah[1]{\vf}N \colon \Ext[1]S{\cotan{\vf}}{N}\xra{\ \cong\ }  \hcm[2]RS{N}\,.
\]
\end{lemma}

\begin{proof}
Consider the construction of the cotangent complex from \ref{ch:cotangent}. It is shown in \cite{Quillen:1968}*{\S 6} that since $\vf$ is surjective, $\HH i {J^2}=0$ for $i\le 1$, and it follows from this that $\Ext[i]{\env RS}{(S\otimes_R P)/J^2}{N}=0$ for $i\le 2$. Hence the exact sequence arising from the triangle (\ref{eq:principalseq}) degenerates into an isomorphism
\[
0\lra \Ext[1]{\env RS}{\cotan{\vf}}{N}\xrightarrow{\dual{(\atiyahclass \vf)}} \Ext[2]{\env RS}{S}{N}\lra 0\,.
\]
The last ingredient is to observe that in (\ref{eq:atiyahdef}) the map $\mu_1$ is also an isomorphism:
\[
\Ext[1]{S}{\cotan{\vf}}{N} = \Hom S {\HH 1{\cotan{\vf}}}{N} =  \Hom {\env RS} {\HH 1{\cotan{\vf}}}{N} = \Ext[1] {\env RS} {\cotan{\vf}}{N}\,.
\]
Since $\atiyah {\vf}N = \dual{(\atiyahclass \vf)}\circ\mu_* $ this completes the proof.
\end{proof}
\end{chunk}

\section{Local complete intersection maps}
\label{se:lci}
This section is entirely about surjective maps of local rings. It contains a proof of Theorem~\ref{intro:thmb} from the introduction. The argument uses Theorem~\ref{th:atiyah-equivariance} and a characterisation of the complete intersection property for a local map in terms of nilpotent elements in its homotopy Lie algebra.

\begin{chunk} 
Let $R$ be a noetherian local ring, $\vf\colon R\to S$ a surjective map, $\ell$ the residue field of $S$, and $F\colonequals \ell\lotimes RS$.  Avramov~\cite{Avramov:1983}*{\S3} attaches to $\vf$ a graded Lie algebra, called the \emph{homotopy Lie algebra} of $\vf$, and denoted $\pi(\vf)$; see also \cite {Avramov/Halperin:1983}. It is a $\ell$-vector subspace of $\Ext F{\ell}{\ell}$ consisting of primitive elements, closed under the commutator, which endows it with the structure of a Lie algebra. We view $\pi(\vf)$ as a subspace of the primitives of $\hcm RS{\ell}$, using \eqref{eq:hh=ext}. 
\end{chunk}

\begin{chunk}
\label{ch:aha}
Let $\vf$ be as above and suppose that $\pdim_RS$ is finite. The contrapositive of \cite{Avramov/Halperin:1987}*{Theorem~C} due to Avramov and Halperin states: If for each $\alpha$ in $\pi^2(\vf)$ and $\beta\in \pi(\vf)$ there exists an integer $s\ge 1$ such that $(\ad \alpha)^s(\beta) = 0$, then $\vf$ is complete intersection.  That is, the kernel of $\vf$ is generated by a regular sequence.
\end{chunk}

This is one of the key ingredients in the proof of the result below. The statement involves the characteristic action of $\hcm RS{\ell}$-action on $\Ext S{\cotan{\vf}}\ell$; see \ref{ch:Yoneda}.

\begin{theorem}
\label{th:main-local-torsion}
Let $R$ be a noetherian  local ring,  $\vf\colon R\to S$  a surjective map, and $\ell$ the residue field of $S$.  If $\pdim_RS$ is finite, and there exists an integer $n$ such that $\hcm[\ges n]RS\ell\cdot  \Ext [1]S{\cotan{\vf}}\ell=0$, then $\vf$ is complete intersection.
\end{theorem}

\begin{proof}
By Theorem~\ref{th:atiyah-equivariance} the Atiyah character $\atiyah{\vf}\ell$ is equivariant with respect to the action of $\hcm RS\ell$. The  hypotheses and Lemma~\ref{le:at2} therefore yield 
\begin{equation}
\label{eq:main-local}
\ad(x)(\alpha) =0 \quad \text{for all $x\in \hcm[\ges n]RS\ell$ and $\alpha\in\hcm[2]RS\ell$\,.}
\end{equation}
We claim that this puts us in a context where the result recalled in~\ref{ch:aha} applies. To see this we exploit the fact that $\pi(\vf)$  consists of primitive elements in $\hcm RS{\ell}$. Thus for any $\alpha\in \pi^2(\vf)$ and $\beta\in \pi^{\ges n}(\vf)$ one gets
\[
\ad(\alpha)(\beta) = [\alpha,\beta] = -[\beta,\alpha] = -\ad(\beta)(\alpha)  = 0\,,
\]
where the first and last equality hold because $\alpha$ and $\beta$ are primitive elements, see \ref{ch:classical-adjoint}; the second one holds since $\alpha$ has degree 2 and the last one holds by \eqref{eq:main-local}. Then for any $s\ge \lceil n/2 \rceil$ and $\gamma\in\pi(\vf)$ one gets that
\[
\ad (\alpha)^{s}(\gamma) = \ad(\alpha)( \ad(\alpha)^{s-1}(\gamma)) =0
\]
because the degree of $\ad (\alpha)^{s-1}(\gamma)$ is $2(s-1) + |\gamma|\ge 2s\ge n$. Thus the result  in \ref{ch:aha} applies to yield that $\vf$ is complete intersection.
\end{proof}

\begin{chunk}
\label{ch:main-local}
Let $\vf\colon R\to S$ be a surjective local map with kernel $I$. Given \eqref{eq:cotan01}, in $\dcat S$ there is a natural morphism 
\[
\cotan{\vf}\lra \susp I/I^2\,.
\]
When $\vf$ is complete intersection, $\pdim_RS$ is finite, the $S$-module $I/I^2$ is free, and the map above is a quasi-isomorphism. Theorem~\ref{th:main-local}  gives a strong converse.

With $\fm$ denoting the maximal ideal of $R$,  one has  $I^2\subseteq \fm I$. Composing the map above with the  surjection $I/I^2\to I/\fm I$ yields a morphism
\begin{equation}
\label{eq:eps-map}
\eps\colon \cotan{\vf} \lra  \susp I/\fm I
\end{equation}
in $\dcat S$. The result below is about this map.

\begin{theorem}
\label{th:main-local}
Let $(R,\fm)$ be a noetherian  local ring and $\vf\colon R\to S$  a surjective map with $\pdim_RS$ finite. If the morphism $\eps\colon \cotan{\vf} \to  \susp I/\fm I$ factors in $\dcat S$  through a perfect complex, then $\vf$ is complete intersection.
\end{theorem}

\begin{proof}
By hypothesis $\eps$ factors as a composition of  morphisms  $\cotan{\vf} \to P \to  \susp  I/\fm I$ in $\dcat S$, where $P$ is a perfect complex. With $\ell$ denoting the residue field of $S$, the map $\Ext S{\eps}{\ell}$ factors as
\[
\begin{tikzcd}
\Ext S{ I/\fm I}\ell \ar[r]  &  \Ext SP\ell \ar[r, "\nu"] & \Ext S{\cotan{\vf}}\ell\,.
\end{tikzcd}
\]
These maps are compatible with the characteristic action of $\hcm RS{\ell}$.  Using $\HH i{\cotan{\vf}}=0$ for $i\le 0$ it is easy to verify that $\Ext[1]S{\eps}\ell$ is an isomorphism, and so $\nu^{1}$ is surjective. As $P$ is perfect, $\Ext[\ges n+1]SP\ell=0$ for some integer $n$, and so
\[
\hcm[\ges n]RS\ell \cdot \Ext[1]SP\ell = 0\,.
\]
Hence the linearity of the maps with respect to the action of $\hcm RS\ell$ and the surjectivity of $\nu^{1}$ imply 
\[
\hcm[\ges n]RS\ell\cdot \Ext[1]S{\cotan{\vf}}\ell=0\,.
\]
Thus Theorem~\ref{th:main-local} applies and yields that $\vf$ is complete intersection.
\end{proof}

\end{chunk}

\begin{chunk}
Let $\vf\colon R\to S$ be a surjective local map with $\pdim_RS$ finite, and set $I\colonequals \Ker(\vf)$. If the second symbolic power $I^{(2)}$ satisfies $I^{(2)}\subseteq \fm I$, then Theorem~\ref{th:main-local} applies to say that if the $S$-module $I/I^{(2)}$ has finite projective dimension, then $I$ is generated by a regular sequence, in which case $I/I^{(2)}$ is free; confer \cite{Smith:2000}.
\end{chunk}

We end this section by sketching  differential graded analogues of the above results. These remarks will not be needed in the sequel, but they may help to put the preceding discussion in context. 

\begin{chunk}
\label{ch:hla}\label{ch:dgcomparison}
Let $\vf\colon R\to S$ be a surjective local map, and $\ell$ the residue field of $S$.  Using a minimal semi-free dg algebra resolution of  $\vf$ in place of a simplicial resolution, one can define a dg version $\mathrm{L^{dg}}(\vf)$ of the cotangent complex;  \cites{Avramov/Herzog:1994,Briggs:2020a}. The dg Atiyah class $\mathrm{dgAt}^\vf\colon S\to \susp \mathrm{L^{dg}}(\vf)$ can be defined analogously and from it the dg Atiyah character. Lemma~\ref{le:at2} and Theorem~\ref{th:atiyah-equivariance} hold for $\mathrm{L^{dg}}(\vf)$.  There is a natural comparison map
\[
\lambda \colon \mathrm{L^{dg}}(\vf)\to \cotan{\vf}\,.
\]
When $R$ contains a field of characteristic zero $\lambda$ is a quasi-isomorphism; otherwise $\cotan{\vf}$ and $\mathrm{L^{dg}}(\vf)$ are typically rather different; see \cite{Avramov:1998}*{pp.~106}.

In any case, the global Atiyah class $\atiyahclass\vf$ factors through ${\rm dgAt}^\vf$ via $\lambda$, so that the Atiyah character $\atiyah{\vf}{\ell}$ factors as
 \[
 \Ext S{\cotan{\vf}}{\ell}\xra{\ \Ext S{\lambda}{\ell}\ } \Ext S{\mathrm{L^{dg}}(\vf)}{\ell} \lra  \hcm RS{\ell}
 \]
where the maps are compatible with the actions of $\hcm RS{\ell}$.  It turns out that the map on the right is bijective onto $\pi(\vf)$. We plan to this elaborate on this, and its connection to Lie algebras of local rings as introduced by Andr\'e~\cite{Andre:1970}, elsewhere. 
\end{chunk}

\section{Cotangent modules}
\label{se:cotangent-modules}
In this section we prove Theorem~\ref{intro:thma}, and the other global results on locally complete intersection maps, announced in the Introduction.

\begin{chunk}
\label{ch:lcidef}
In their full generality, locally complete intersection maps were defined by Avramov in \cite{Avramov:1999}. We start with a local homomorphism $\vf\colon R\to S$. By \cite{Avramov/Foxby/Herzog:1994}*{1.1}, the composition of $\vf$ with the completion map $S\to \wh S$ at the maximal ideal of $S$ admits a \emph{regular factorisation}:
\begin{equation}\label{eqn:factorisation}
R\xra{\ \dot\vf\,} R'\xra{\ \vf'}{\wh S}\,,
\end{equation}
where $\dot\vf$ is a flat local map whose closed fibre is regular, and $\vf'$ is surjective.  One says that $\vf$ is \emph{complete intersection} if in some, equivalently, any, regular factorisation (\ref{eqn:factorisation}) the kernel of $\vf'$ is generated by a regular sequence; see~\cite{Avramov:1999}*{3.3}.

A map $\vf\colon R\to S$ of noetherian rings is \emph{locally complete intersection} if for each $\fp$ in $\spec S$ the localisation $\vf_\fp\colon R_{\fp\cap R}\to S_\fp$ is complete intersection.

When $\vf$ is essentially of finite type the definition is already in \cite{Grothendieck:1964}.  In this context there is no need to complete $S$ in (\ref{eqn:factorisation}), and  one may take $\dot\vf$ to be smooth.
\end{chunk}

\begin{chunk}
\label{ch:syz}
Let $S$ be a ring and $M$ a complex $S$-modules such that $\HH iM=0$ for $i\ll 0$. Let $P\xra{\sim} M$ be a projective resolution of $M$; thus $P$ is a complex of projective $S$-modules with $P_i=0$ for $i\ll 0$.  If $Q$ is another such resolution of $M$, then for each integer $n$, there is an quasi-isomorphism
\[
P_{\ges n} \oplus \susp^n Q' \xra{\ \sim\ } Q_{\ges n} \oplus \susp^n P'
\]
where $P'$ and $Q'$ are projective $S$-modules; see \cite{Avramov/Iyengar:2007}*{1.2}. In particular, one has an isomorphism of  $S$-modules
\[
\HH n{P_{\ges n}} \oplus Q' \cong  \HH n{Q_{\ges n}} \oplus P'\,.
\]
When $M$ is a module, the $S$-module $\HH n{P_{\ges n}}$ is its $n$th syzygy module, and the discussion above says nothing more than that a syzygy module is well-defined, up to projective summands.
\end{chunk}

\begin{chunk}\label{ch:cotandef}
Let $\vf\colon R\to S$ be a map of noetherian rings and $\cotan{\vf}$  its cotangent complex. Choose a projective resolution $P$ of $\cotan{\vf}$ and for each integer $n$ set 
\[
\cotam n{\vf}\colonequals \HH n{P_{\ges n}}\,.
\]
We call this $S$-module the \emph{$n$th cotangent module} of $\vf$; by \ref{ch:syz} it is independent of $P$, up to projective summands. In particular, the  flat dimension, or the projective dimension, of $\cotam n{\vf}$ is well-defined, and an invariant only of $\vf$. Evidently $\cotam 0{\vf}\cong \Omega_{S/R}$, and in the case of a surjective homomorphism, when $S=R/I$ one has
\[
\cotam 0{\vf} = 0 \quad\text{and}\quad \cotam 1{\vf}\cong I/I^2\,.
\]
Therefore the cotangent modules can be considered as higher analogues of both the module of differentials and the conormal module. These objects were introduced by Avramov and Herzog \cite{Avramov/Herzog:1994}; in their context $R$ is a field of characteristic zero and $S$ is positively graded.

For  $\fp$ in $\spec S$ there is a quasi-isomorphism ${\cotan \vf}_\fp\simeq \cotan {\vf_\fp}$ by \cite{Andre:1974}*{IV.59~\&~V.27}. It follows that  ${\cotam n{\vf}}_\fp\cong \cotam n{\vf_\fp}$ up to projective $S_\fp$-module summands.

\end{chunk}

\begin{lemma}
\label{le:cotam-fdim}
Let $R\xra{\dot\vf} R'\xra{\vf'}S$ be maps of noetherian rings such that $\dot\vf$ is smooth. Then for any integer $n\ge 1$ the flat dimension of $\cotam n{\vf'\dot\vf}$ is equal to that of $\cotam n{\vf'}$. 
\end{lemma}

\begin{proof}
Setting $\vf\colonequals \vf'\dot\vf$ yields the Jacobi-Zariski exact triangle 
\[
S\lotimes{R'} \cotan{\dot\vf} \lra \cotan{\vf} \lra \cotan{\vf'} \lra
\]
in $\dcat S$; see \cite{Quillen:1968}*{5.1}. Let $Q$ be a projective resolution of $ \cotan{\dot\vf}$, and $P$ be a projective resolution of $\cotan{\vf}$. Then $S\lotimes{R'} \cotan{\dot\vf} \lra \cotan{\vf}$ may be represented by a map of complexes $S\otimes_{R'}Q\to P$, and we may take $P'=\cone(S\otimes_{R'}Q\to P)$ as a projective resolution of $\cotan{\vf'}$. Having done this the truncated resolutions fit into a triangle
\[
S\otimes_{R'} Q_{\ges n-1}\lra P_{\ges n}\lra P'_{\ges n}\lra\,.
\]
Taking homology we obtain an exact sequence
\[
\HH{n}{S\lotimes{R'} \cotan{\dot\vf}}\lra \cotam n{\vf} \lra \cotam n{\vf'}\lra S\otimes_{R'} \cotam {n-1}{\dot\vf}\lra 0 \,.
\]
Since $\dot\vf$ is smooth, $\cotan{\dot\vf}$ has flat dimension $0$ by \cite{Andre:1974}*{XVI.17}, and hence the same is true of the $S$-complex $S\lotimes{R'} \cotan{\dot\vf}$. In particular $\HH{n}{S\lotimes{R'} \cotan{\dot\vf}}=0$ and the module $S\otimes_{R'}  \cotam {n-1}{\dot\vf}$ is flat, so the exact sequence is witness to the equality $\fdim_S \cotam n{\vf}=\fdim_S \cotam n{\vf'}$.
\end{proof}

When $\vf$ is locally complete intersection, the flat dimension of $\cotan{\vf}$ is at most one, and the $S$-modules $\cotam{n}{\vf}$ are flat for $n\ge 1$. The result below provides a strong converse, at least for maps essentially of finite type.  It generalises \cite{Avramov/Herzog:1994}*{Theorem~2.4} due to Avramov and Herzog, and answers a problem they pose in \cite{Avramov/Herzog:1994}*{\S 3}.

\begin{theorem}
\label{th:cotam}
Let  $\vf\colon R\to S$ be a map of  noetherian rings, essentially of finite type and locally of finite flat dimension. If for some integer $n\ge 1$ the $S$-module $\cotam n{\vf}$ has finite flat dimension, then $\vf$ is locally complete intersection.
\end{theorem}

\begin{proof}
By \ref{ch:cotandef} we may assume $\vf$ is local. Since $\vf$ is essentially of finite type it can be factored as $R\xra{\dot\vf} R'\xra{\vf'}S$, with $\dot\vf$ smooth and with $\vf'$ surjective; see \ref{ch:lcidef}. Lemma~\ref{le:cotam-fdim} yields that $\cotam n{\vf'}$ has finite flat dimension. It follows from the fact that $\vf$ has finite flat dimension and $\dot\vf$ is smooth that $\vf'$ has finite flat dimension as well; see, for instance, \cite{Avramov/Foxby/Herzog:1994}*{Lemma~3.2}. Hence we may replace $\vf$ with $\vf'$, and assume that it is surjective. 

The homology modules $\HH i{\cotan{\vf}}$ are finitely generated and equal to $0$ for $i\le 0$. Thus there exists a resolution $P\xra{\sim}\cotan{\vf}$ consisting of finitely generated projective modules, with $P_i=0$ for $i\le 0$. In particular, we can assume $\cotam n{\vf}$ is finitely generated, so the hypothesis gives that its projective dimension is finite. Moreover $\HH 1P=I/I^2$, where $I=\Ker(\vf)$, so  there is a natural morphism of complexes $P \to \susp I/I^2$. Observe that this factors through the quotient complex
\[
0\lra \cotam n{\vf} \lra P_{n-1} \lra \cdots \lra  P_1\to 0\,.
\]
Since the projective dimension of $\cotam n{\vf}$ is finite, and each $P_i$ is projective, the complex above is perfect. Thus in $\dcat S$ the  morphism $\cotan {\vf}\to \susp I/I^2$ factors through a perfect complex, and hence the same is true of the morphism 
\[
\cotan{\vf}\lra \susp (I/\fm I)
\]
where $\fm$ is the maximal ideal of $R$; see \ref{ch:main-local}. We can now apply Theorem~\ref{th:main-local} to conclude that $\vf$ is complete intersection.
\end{proof}

In Theorem \ref{th:cotam}, when $S=R/I$, the case $n=1$ says that $I$ is a locally complete intersection ideal as soon as the conormal module $\cotam 1 \vf=I/I^2$ has finite projective dimension. This was a conjecture of Vasconcelos \cite{Vasconcelos:1978}, resolved by the first author in \cite{Briggs:2020a}. Even in this special case, the proof given here is new. We will see soon that Theorem~\ref{th:cotam} also yields a new proof of a theorem of Avramov \cite{Avramov:1999}.

What happens in the case $n=0$ is still unclear. Supposing that $R$ is a field of characteristic zero, Vasconcelos conjectured that if $\cotam 0 \vf = \Omega_{S/R}$ has finite projective dimension over $S$, then $S$ is a locally complete intersection ring \cite{Vasconcelos:1978}. In general this remains open, but see \cite{Avramov/Herzog:1994}*{2.7} and \cite{Briggs:2020a}*{\S3.4}.

\begin{chunk}\label{ch:twoflat}
The flatness of the cotangent modules can be detected through the \emph{Andr\'e-Quillen homology functors}, namely 
\[
\aqh nRSM \colonequals \HH{n}{ \cotan{\vf}\lotimes SM}\,.
\]

Indeed, we can take a projective resolution $P$ of $\cotan{\vf}$, and consider the natural map $\cotam n{\vf}\to P_{n-1}$.  For an $S$-module $M$, the hypothesis $\aqh nRSM =0$ means exactly that the induced map
\[
\cotam n{\vf}\otimes_SM \lra P_{n-1}\otimes_SM
\]
is injective. In particular, if $\aqh nRS- =0$  on the category of $S$-modules,  then $\cotam n{\vf}$ is a \emph{pure submodule} of $P_{n-1}$.  Because of the exact sequence
\[
0\lra\cotam n{\vf}\lra P_{n-1}\lra \cotam {n-1}{\vf}\lra 0\,,
\]
the flatness of $P_{n-1}$ would then be inherited by both  $\cotam n{\vf}$ and $\cotam {n-1}{\vf}$.
\end{chunk}

The next result implies immediately Quillen's conjecture \cite{Quillen:1968}*{5.6} that $\cotan \vf$ can only have finite flat dimension if $\vf$ is locally complete intersection, as proven originally by Avramov \cite{Avramov:1999}*{Theorems 1.4, 1.5}.  It also answers, in the affirmative, a question posed by him~\cite{Avramov:1984}*{Remark~2}.

\begin{corollary}
\label{ch:Quillen}
Let $\vf\colon R\to S$ be a map of noetherian  rings, locally of finite flat dimension. If $\aqh nRS-=0$ for an $n\ge 1$, then $\vf$ is locally complete intersection.
\end{corollary}

\begin{proof}
We may assume that $\vf$ is local by flat base change, and then we may take a regular factorisation of $R\to \wh S$ as in (\ref{eqn:factorisation}). It follows from \cite{Avramov:1999}*{1.1} for $n=1$, and from \cite{Avramov:1999}*{1.7} for $n\geq 2$, that if $\aqh nRS-=0$ then as well $\aqh n{R'}{\wh S}-=0$. Hence we may assume $\vf$ is surjective.

By the observation in \ref{ch:twoflat}, the vanishing of $\aqh nRS-$ implies that both $\cotam n{\vf}$ and $\cotam {n-1}{\vf}$ are flat. So $\vf$ is complete intersection by Theorem~\ref{th:cotam}.
\end{proof}

We note that Theorem~\ref{th:cotam} is stronger than Corollary~\ref{ch:Quillen}: The hypotheses of the later yield that \emph{two} consecutive cotangent modules are flat, while the former theorem only requires a single cotangent module to have finite flat dimension.

\begin{chunk}
The analogue of Corollary~\ref{ch:Quillen} for the homotopy Lie algebra says that for a surjective local homomorphism $\vf\colon R\to S$ of finite flat dimension, if $\pi^n(\vf)=0$ for some $n\geq 2$ then $\vf$ is complete intersection. This was proven by  Halperin in the case $R$ is regular \cite{Halperin:1987}, and by Avramov in general \cite{Avramov:1999}. The arguments of this paper can also be adapted to prove this result; indeed, the analogue of Theorem~\ref{th:cotam} works equally well using syzygies of the dg cotangent complex; see \ref{ch:dgcomparison}.
\end{chunk}

\begin{chunk}
Corollary~\ref{ch:Quillen} is true as well for the Andr\'e-Quillen cohomology functors, since $\aqch nRS-=0$ implies $\aqh nRS-=0$. In fact, the vanishing of $\aqch nRS-$ implies projectivity of the higher cotangent modules $\cotam n{\vf}$ and $\cotam {n-1}{\vf}$, which is a stronger condition than their flatness.
\end{chunk}

\begin{chunk}\label{ch:geometry}
The results above can be restated geometrically. Let $f\colon X\to Y$ be a morphism of locally noetherian schemes. Illusie \cite{Illusie:1971} associates to $f$ a cotangent complex $\cotan f$ in the derived category of quasi-coherent sheaves on $X$, from which we obtain the cotangent homology functors 
\[
\aqh n YX- \colonequals \mathrm{H}^{-n}(\cotan f\lotimes X -)\,.
\]
By \cite{Illusie:1971}*{3.1.1} the vanishing of the cotangent homology functors can be detected on an open affine cover of $X$. Hence, it follows immediately from Corollary~\ref{ch:Quillen} that if $f$ is  locally of finite flat dimension and $\aqh n YX- =0$ for some $n\geq 1$ then $f$ is locally complete intersection.

Now assume further that $f$ is essentially of finite type, and that $X$ has enough vector bundles (locally free sheaves); see \cite{Orlov:2004}*{1.2}. The cotangent complex $\cotan f$ can then be represented as a bounded below complex $P$ of vector bundles. We define the $n$th cotangent sheaf of $f$ to be
\[
\cotam n f\colonequals \coker(\partial\colon P_{n+1}\to P_{n})\,.
\]
In this context we cannot assert that $\cotam n f$ is well-defined up to vector bundle summands. Instead we consider $\cotam n f$ as an object in the \emph{singularity category} $\dsing X$ of Buchweitz \cite{Buchweitz:1986} and Orlov \cite{Orlov:2004}. Indeed, if $P\to P'$ is a quasi-isomorphism to another bounded below complex of vector bundles representing $\cotan f$, then the induced map on the relevant cokernels  has a perfect cone by \ref{ch:cotandef}, and is therefore an isomorphism in $\dsing X$. Finally, applying Theorem~\ref{th:cotam} locally yields the following statement:

\begin{theorem}
If $f\colon X\to Y$ is a morphism of locally noetherian schemes, essentially of finite type and locally of finite flat dimension, and if $\cotam n f \simeq 0$ in $\dsing X$ for some $n\geq 1$ then $f$ is locally complete intersection. \qed
\end{theorem}

\end{chunk}


\end{document}